\documentclass{amsart}
\usepackage{amssymb}
\usepackage{amsfonts}
\usepackage{amstext}
\usepackage{algorithmic}
\usepackage{algorithm}
\usepackage{graphicx}
\usepackage{epstopdf}
\usepackage[all]{xy}

\numberwithin{equation}{section}
\newtheorem{theorem}{Theorem}[section]
\newtheorem{lemma}[theorem]{Lemma}

\newtheorem{corollary}[theorem]{Corollary}
\newtheorem*{theoremA}{Theorem A}
\newtheorem*{theoremB}{Theorem B}
\newtheorem*{theoremC}{Theorem C}

\newtheorem{conjectures}[theorem]{Conjectures}

\newtheorem{facts}[theorem]{Facts}
\theoremstyle{definition}
\newtheorem{definition}[theorem]{Definition}
\newtheorem{example}[theorem]{Example}

\theoremstyle{remark}
\newtheorem{remark}[theorem]{\bf{Remark}}

\newtheorem{convention}[theorem]{\bf{Convention}}

\newtheorem{notconv}[theorem]{\bf{Convention-Notation}}

\newtheorem*{prooftheoremA}{Proof of Theorem A}
\newtheorem*{prooftheoremB}{Proof of Theorem B}
\newtheorem*{prooftheoremC}{Proof of Theorem C}

\newcommand{\Fcal}{{\mathcal{F}}}
\newcommand{\Gcal}{{\mathcal{G}}}

\newcommand{\Lcal}{{\mathcal{L}}}

\newcommand{\Ncal}{{\mathcal{N}}}

\newcommand{\Rcal}{{\mathcal{R}}}

\newcommand{\ra}{{\triangleleft}}

\newcommand{\tens}{\otimes}

\renewcommand{\ra}{{\triangleleft}}

\newcommand{\id}{\mathord{\mathrm{id}}}

\newcommand{\Sym}{\mathop{\mathrm{Sym}}}

\begin{document}

\title[Square-free solutions of YBE]{{The braided group of a square-free solution of the Yang-Baxter equation and its group algebra}
\keywords{Yang-Baxter; Quantum groups; Braided group;  Group rings; Finite generation, Finite presentability}
\subjclass[2000]{Primary  16T25, 16S15. 81R50, 16W35, 16S34, 16W22, 16S37, 20B35}}
 	   	
\thanks{The author was partially supported by the Max Planck Institute for Mathematics in the Sciences (MiS) in Leipzig}

\author{Tatiana Gateva-Ivanova}
\address{American University in Bulgaria,
Blagoevgrad 2700, Bulgaria}
\email{ tatyana@aubg.edu}
\date{\today}
\begin{abstract}
Set-theoretic solutions of the Yang--Baxter equation form a meeting-ground
of mathematical physics, algebra and combinatorics. Such a solution $(X,r)$
consists
of a set $X$
and a bijective map $r:X\times X\to X\times X$ which satisfies the braid
relations.
In this work we study the braided group $G=G(X,r)$ of an involutive square-free  solution $(X,r)$ of finite order $n$ and cyclic index $p=p(X,r)$ and the group algebra $\textbf{k} [G]$ over a field $\textbf{k}$. We show that
$G$ contains a $G$-invariant normal subgroup $\Fcal_p$ of finite index $p^n$,  $\Fcal_p$ is isomorphic to
the free abelian group of rank $n$. We describe explicitly the quotient braided group $\widetilde{G}=G/\Fcal_p$ of order $p^n$ and show that $X$ is embedded in $\widetilde{G}$.
We prove that the group algebra $\textbf{k} [G]$  is a free left (resp. right) module of finite rank $p^n$ over
its commutative subalgebra $\textbf{k}[\Fcal_p]$ and give an explicit free basis. The center of $\textbf{k} [G]$ contains the
subalgebra of symmetric polynomials in $\textbf{k} [x_1^p, \cdots, x_n^p]$.  Classical results on group rings imply that $\textbf{k}[G]$ is a left (and right) Noetherian domain of finite global dimension.
\end{abstract}
\maketitle

\section{Introduction}
Let $V$ be a vector space over a field $k$. It is well-known that the
``Yang--Baxter equation'' on a linear map $R:V\tens V\to V\tens V$,
the equation
\[ R_{12}R_{23}R_{12}=R_{23}R_{12}R_{23}\]
(where $R_{i,j}$ denotes $R$ acting in the $i,j$ place in $V\tens V\tens V$),
gives rise to a linear representation of the braid group on tensor powers of
$V$. When $R^2=\id$ one says that the solution is involutive, and in this
case one has a representation of the symmetric group on tensor powers.
A particularly nice class of solutions is provided by set-theoretic solutions,
where $X$ is a set and $r:X\times X\to X\times X$ obeys similar relations on
$X\times X\times X$, \cite{Dri}. Of course, each such solution extends linearly to $V=kX$
with matrices in this natural basis having only entries from 0,1 and many
other nice properties.
During the last two decade the study of
set-theoretic solutions and related
structures notably intensified, a relevant selection of works for the
interested reader
is \cite{GIVB, ESS, CJO14, Dehornoy,  LYZ, GI04, Ru05, Takeuchi, GIM08, GI12, GIC, CGIS17, CGIS18, GI18, GI18a, BCV, Smok}, and references therein.
 We shall use the terminology,  notation and some results from
\cite{GI04, GIM08, GI12, GIC, GI18}.

 \begin{definition}
Let $X$ be a nonempty set (not necessarily finite) and let
$r: X\times X \longrightarrow X\times X$ be a bijective map.
We use notation $(X, r )$ and refer to it as \emph{a quadratic set}.
The image
of $(x,y)$ under $r$ is presented as
$r(x,y)=({}^xy,x^{y})$.
This formula defines a ``left action'' $\Lcal: X\times X
\longrightarrow X,$ and a ``right action'' $\Rcal: X\times X
\longrightarrow X,$ on $X$ as:
\[
 \Lcal_x(y)={}^xy, \quad \Rcal_y(x)= x^{y}, \quad\quad \forall x, y \in X.
\]
(i) $r$ is \emph{nondegenerate}, if
the maps $\Lcal_x$ and $\Rcal_x$ are bijective,  $\forall\; x\in X$.
(ii) $r$ is \emph{involutive} if $r^2 = id_{X\times X}$.
(iii) $(X,r)$ is called \emph{square-free} if $r(x,x)=(x,x)$,  $\forall\;  x\in
X.$
(iv) $r$ is \emph{a set-theoretic solution of the
Yang--Baxter equation} (YBE) if the braid relation
\[r^{12}r^{23}r^{12} = r^{23}r^{12}r^{23}\]
holds in $X\times X\times X,$  where  $r^{12} = r\times\id_X$, and
$r^{23}=\id_X\times r$. In this
case, see \cite{ESS} $(X,r)$ is called \emph{a braided set}.
(v) A braided set $(X,r)$ with $r$ involutive is called \emph{a
symmetric set}.
\end{definition}
\begin{convention}
\label{convention1}
In this paper "\emph{a solution}" means \emph{a nondegenerate}
\emph{symmetric set} $(X,r)$, where $X$  is a set of arbitrary
cardinality. We shall also refer to it as "\emph{a symmetric set}", keeping
the convention
that we consider \emph{only }\emph{nondegenerate} symmetric sets.
\end{convention}
\begin{definition}
\label{algobjectsdef}
To each quadratic set $(X,r)$  we associate \emph{canonical algebraic
objects}  \emph{generated by} $X$ and with \emph{quadratic defining relations}
$\Re=\Re(X,r)$, given by  \[xy=zt \in \Re\; \text{iff} \;r(x,y) = (z,t)\;\text{and}\;(x,y)\neq (z,t).\]
(i) The monoid $ S =S(X, r) = \langle X; \Re \rangle$,
with a set of generators $X$ and a set of defining relations $
\Re,$ is called \emph{the monoid associated with $(X, r)$}.
(ii) The \emph{group $G=G(X, r)$ associated with} $(X, r)$ is
defined as $G=G(X, r)={}_{gr} \langle X; \Re \rangle$.
Each element $a \in G$ can be presented as a monomial
$a = \zeta_1\zeta_2 \cdots \zeta_m,\quad \zeta_i \in X
\bigcup X^{-1}.$
We shall consider  \emph{a reduced form of} $a$, that is a
presentation with minimal length $m$. \emph{By convention}, \emph{length
of} $a$,
denoted by $|a|$ means \emph{the length of a reduced form of} $a$.
(iii) For arbitrary fixed field $\textbf{k}$, \emph{the $\textbf{k}$-algebra
associated with} $(X ,r)$ is defined as $A(\textbf{k},X,r) = \textbf{k}\langle X ;
\Re \rangle$. $A(\textbf{k},X,r)$ is a quadratic algebra isomorphic to the monoidal algebra
$\textbf{k}S(X,r)$.
(iv) Furthermore, to each nondegenerate braided set $(X,r)$ we also associate \emph{a
permutation group},
denoted $\Gcal= \Gcal(X,r)$, see Definition \ref{Gcaldef}.
 If $(X,r)$ is a solution then $G=G(X, r)$ has a canonical structure of an involutive braided group, see Facts
 \ref{factLYZ}.
In the paper we shall refer to $G$ as "\emph{the (involutive) braided group
associated to} $(X,r)$", see Definition \ref{assocBraidedGroupdef}
\end{definition}
If $(X,r)$  is a braided
set, and $G=G(X,r)$, then
the assignment $x \longrightarrow \Lcal_x$ for
$x\in X$ extends canonically to a group homomorphism
$\Lcal: G \longrightarrow \Sym(X)$,
which defines  the \emph{canonical left action} of $G$  on the set  $X$.
Analogously,  there is a \emph{canonical right action} of
$G$  on $X$, \cite{ESS}.
\begin{definition}
\label{Gcaldef}
(i) The image $\Lcal(G(X,r))$ is a subgroup of $\Sym(X)$, that is, a permutation group. We
denote it by $\Gcal=\Gcal(X,r)$, and call it \emph{the permutation group (of left actions) of}
$(X,r)$, \cite{GI04, GIC}.
$\Gcal$ is generated by the set $\{\Lcal_x \mid x \in X\}$.

(ii) If $(X, r)$ is a finite solution then \emph{the least
common multiple of all orders of permutations} $\Lcal_x \in \Gcal$, $x \in X$, is called  \emph{the cyclic degree of} $(X,r)$ and denoted by $p=p(X,r)$, see \cite{GI04}, Def 3.17.
One has $(\Lcal_x)^p= id_X$ for all $x \in X$, and $p$ is the minimal integer with this property.
\end{definition}

\begin{example}
\label{trivialsolex} For arbitrary
set $X$, $|X|\geq 2,$ denote by $\tau_X = \tau$ the flip map
$\tau(x,y)= (y,x)$ for all $x,y \in X.$ Then  $(X, \tau)$ is a square-free
solution called \emph{the trivial solution} on $X$. It is clear that an involutive quadratic set
$(X,r)$ is the trivial solution if and only if ${}^xy =y$,  for all $x,y \in X,$ or equivalently $\Lcal_x= \id_X$ for all $x\in X$. In
this case $S(X,r)$ is the free
abelian monoid generated by $X$, $G(X,r)$ is the free abelian group, $A(\textbf{k},X,r)$ is
the algebra of commutative polynomials in $X$, and $\Gcal(X,r) =
\{\id_X\}$ is the trivial group.
\end{example}

The results of \cite{GIVB}
show the close relation among various
mathematical notions and theories: the set-theoretic solutions of the
Yang-Baxter equation, the semigroups of I-type, the skew binomial polynomial rings,
introduced by the author in \cite{GI96}, and the theory of Bieberbach groups.
We shall use the following results.
\begin{theorem}
\label{theorem1}
\begin{enumerate}
\item[(I)] \cite{GIVB}.
Every $X$-generated binomial skew-polynomial ring $A= A(k, X, r) = \textbf{k}\langle X \mid \Re\rangle $ defines a square-free solution $(X,r)$. In this case
the map $r$ is defined canonically via the
 set of quadratic defining relations $\Re$ of $A$.
 Moreover, $A$ is an Artin-Schelter regular Noetherian domain of global dimension $n$.
\item[(II)] \cite{GI04}, \cite{Ru05}.  Conversely,  suppose $(X,r)$ is a finite square-free  solution then the set $X$
   can be ordered $X = \{x_1< x_2  < \cdots < x_n\}$, so that the $\textbf{k}$-algebra $A= A(\textbf{k}, X, r) = \textbf{k}\langle X \mid \Re\rangle $ over
   arbitrary field $\textbf{k}$, is
\emph{a skew polynomial ring with binomial relations} in the sense of \cite{GI96}, or shortly \emph{a binomial skew-polynomial rings} (see Def. \ref{binomialringdef}).
\end{enumerate}
\end{theorem}

The main results of the paper are Theorems \textbf{A}, \textbf{B}, and \textbf{C} which are proven under the following  assumptions, notation and convention.
\begin{notconv}
\label{notation}
$(X, r)$ is \emph{a square-free solution of finite order}  $|X|= n$, and with \emph{a cyclic degree} $p=p(X,r)$, see Def \ref{Gcaldef}. $\textbf{k}$ is a field. We fix an enumeration
$X = \{x_1< x_2  < \cdots < x_n\}$ on $X$, so that the associated Yang-Baxter $\textbf{k}$-algebra $A= A(\textbf{k}, X, r) = \textbf{k}\langle X \mid \Re\rangle $ is
\emph{a skew polynomial ring with binomial relations}, $\Ncal$ is the normal $\textbf{k}$-basis of $A$,  see Definition \ref{binomialringdef}, and Theorem \ref{theorem1}.
$G=G(X,r)$,  $\Gcal= \Gcal(X,r)$, and $S=S(X,r)$ are the associated groups and monoid, see Def \ref{algobjectsdef}. $(G,r_G)$, and
$(\Gcal, r_{\Gcal})$ are the corresponding involutive braided groups (see Facts \ref{factLYZ}).
We introduce the sets
\begin{equation}
\label{eqY}
\begin{array}{lll}
Y   &:= & \{x_1^{\alpha_1}\cdots x_n^{\alpha_n} \mid 0 \leq \alpha_1 \leq p-1\} \subset \Ncal;\\
X_p &:= & \{x_1^p,  x_2^p, \cdots, x_n^p\} \subset \Ncal ;\\
N_p &:= & [X_p] \subseteq S   \quad\text{is \emph{the submonoid of} S \emph{generated by}}\; X_p.
\\
\Fcal_p &:= & {}_{gr}[x_1^p, \cdots, x_n^p] \leq G \quad\text{is \emph{the subgroup of} G \emph{generated by}}\; X_p.\\
A_p &:=& \textbf{k}N_p \;\text{is \emph{the semigroup algebra of}}\; N_p, \quad\text{equivalently}, \\
A_p && \text{is the subalgebra of} \;A, \;\text{generated by}\; X_p.
\end{array}
\end{equation}
Each element $W \in N_p$ will be considered in its normal form, so we identify (as sets) $N_p$ and $\Ncal_p := N_p\bigcap \Ncal$.
\end{notconv}

\begin{theoremA}
Let $(X, r)$ be a square-free solution of finite order  $|X|= n$ and cyclic degree $p=p(X,r)$, let $G=G(X,r)$.
\begin{enumerate}
\item
\label{Thm_Fcal41} Every element $u\in G$ has unique presentation as
\[u =y.W, \quad \text{where} \; y \in Y, \; W\in \Fcal_p.\]
\item
\label{Thm_Fcal42} In particular,
for every pair
$y,z \in Y$ there exist
uniquely determined $t \in Y$, and $W \in \Fcal_p$, such that the equality
\begin{equation}
\label{producteq}
y.z = t.W \;\text{holds in}\; G.
\end{equation}
\item
\label{Thm_Fcal43}
 $G$ splits as a union of $p^n$ disjoint subsets
\begin{equation}
\label{disjuneq}
G = \bigcup_{y \in Y} y \Fcal_p, \quad \text{where}\;\; (y \Fcal_p) \bigcap (t \Fcal_p) = \emptyset, \; \forall \;  y \neq t, \; y, t \in Y.
\end{equation}
\end{enumerate}
\end{theoremA}

\begin{theoremB}
\label{Thm_Fcal}
Let $(X, r)$ be a square-free solution of finite order  $|X|= n$ and cyclic degree $p=p(X,r)$, let $(G, r_G)$ be the associated braided group,
where $G=G(X,r)$.
\begin{enumerate}
\item
\label{Thm_Fcal2a}
$\Fcal_p $ is an ideal of $(G, r_G)$ and there is a chain of ideals $\Fcal_p \;\triangleleft \;\Gamma\;\triangleleft\; G$, in $G$.
\item
\label{Thm_Fcal06a}
The quotient group $\widetilde{G}= G/\Fcal_p$ is an involutive braided group of finite order $p^n$.
More precisely, the canonical projection
\[\pi: G \longrightarrow \widetilde{G}= G/\Fcal_p, \quad g = y W \mapsto \tilde{g}=\tilde{y}\] induces the structure of an involutive braided group
on the finite set \[\widetilde{G} = \tilde{Y} :=\{\tilde{y}\mid y \in Y\},\]
such that the induced operation "$\cdot$" on $\tilde{Y}$ is well-defined via (\ref{producteq}) as:
\begin{equation}
\label{operationeq}
\tilde{y}\cdot\tilde{z}:= \tilde{t},\; \text{where}\;\; y.z = t. W, \; \;\text{holds in}\;\; G, \; y, z, t \in Y,\;\; W \in \Fcal_p.
\end{equation}
The  restriction $\pi_{|Y} : Y \longrightarrow \tilde{Y} =\widetilde{G}$ is a bijective map of sets of order $p^n$.
\item
\label{Thm_Fcal06b}
The map  $X \longrightarrow \tilde{X} \subseteq \tilde{G}$ is an embedding of $X$ into the braided group $\tilde{G}$,
we shall identify the sets $X$ and $\tilde{X}$.  Thus, $X$ is a set of generators of the braided group $\tilde{G}$ of order $p^n$.
\item
\label{Thm_Fcal06c}
There is a canonical epimorphism of braided groups $\tilde{G}\twoheadrightarrow\Gcal \simeq G /\Gamma$, so the order $|\Gcal|$
divides $p^n$.  Moreover, if $p$ is a prime number, then
$\Gcal$ is a $p$-group.
\end{enumerate}
\end{theoremB}

\begin{theoremC}
\label{Thm_FcalB}
Let $(X, r)$ be a square-free solution of finite order $|X|= n$ and cyclic degree $p=p(X,r)$, $G=G(X,r)$.
Let $\textbf{k}[G]$ and $\textbf{k}[\Fcal_p]$ be the corresponding group algebras over a field $\textbf{k}$.
The following conditions are in force.
\begin{enumerate}
\item
  \label{Thm_Fcal4}
$\textbf{k}[\Fcal_p]$ is a commutative Noetherian domain. $\textbf{k}[G]$ is a free left (resp. right) module of finite rank
$p^n$
over the algebra $\textbf{k}[\Fcal_p]$
with a free basis the set $Y$.
  \item
 \label{Thm_Fcal2}
 Let $\mathfrak{S}$ be the subalgebra of $\textbf{k}[G]$ (and of $A$) consisting of all symmetric polynomials $f(x_1^p, \cdots, x_n^p) \in A_p$ in n variables $x_1^p, \cdots x_n^p$.
  Then $\mathfrak{S}$ is an $n$-generated subalgebra contained in the center of $\textbf{k}[G]$ (and of $A$).
 \item
  \label{Thm_Fcal5}
 $\textbf{k}[G]$ is a  left and a right Noetherian domain. Moreover, $\textbf{k}[G]$ has finite global dimension, whenever $\textbf{k}$ is a field of characteristic $0$.
 \end{enumerate}
 \end{theoremC}
 The paper is organized as follows.
In Section 2 we recall some basic definitions and results which will be used throughout the paper.
In Sec 3.1 we give a short introduction to square-free solutions $(X,r)$ and their algebraic objects. In particular,
we present some useful formulae.
 In Section 3.2. we study the braided group $G(X,r)$ and prove Theorems A and B.
The group ring $\textbf{k} [G]$ over a field $\textbf{k}$ is studied in Section 3.3, where we prove theorem C.

\section{Preliminaries on braided groups}

 In this section $(X,r)$ denotes a nondegenerate braided set, where the set $X$ has arbitrary cardinality.
 Lu, Yan and Zhu, \cite{LYZ} proposed a general way of constructing set-theoretical solutions of the
 Yang-Baxter equation using braiding operators on groups, or essentially,  the matched pairs of groups.
In his  survey  Takeuchi \cite{Takeuchi}  gave an introduction to
 the ESS-LYZ theory, reviewing the main results in \cite{LYZ,ESS},  from a matched pair of groups point of view, among them a good way
 to
 think about the
 properties of the group $G(X,r)$ universally generated from $(X,r)$. In particular, it is known that the group $G(X,r)$ is itself a
 braided set
 in an induced manner.

\begin{definition}\cite{LYZ}
\label{def_SymmetricGr}
\emph{A braided group} is a pair  $(G, \sigma)$, where $G$ is a group and
\begin{equation}
\label{leftacteq}
\sigma: G\times G \longrightarrow G\times G, \quad \sigma(a, u) = ({}^au,
a^u)
\end{equation}
is a bijective map, such that the left and the right actions induced by $\sigma$ via (\ref{leftacteq}) satisfy the following
conditions for all  $a,b, u,v\in G$:
\[\begin{array}{llllll}
\textbf{ML0}:&{}^a1 =1, \; {}^1u =u,\quad & \textbf{ML1}:&{}^{ab}u ={}^a{({}^bu)},\quad &\textbf{ML2}:&{}^a{(u.v)} =({}^au)({}^{a^u}v),
\\
\textbf{MR0}:& 1^u =1,
\; a^1 =a, & \textbf{MR1}:& a^{uv} =(a^u)^v,
&\textbf{MR2}:& (a.b)^u
=(a^{{}^bu})(b^u),
\end{array}
\]
and
\emph{the compatibility condition} \textbf{M3}:
\[\textbf{M3}: \quad uv =({}^uv)(u^v), \;\text{is an equality in}\; G.
\]
If the map $\sigma$ is involutive ($\sigma ^2 = id_{G \times G}$) then $(G,
\sigma)$ is \emph{an involutive braided group}, also called "\emph{a symmetric group}",  see \cite{Takeuchi}, and \cite{GI18}. The notion of a "symmetric group", was suggested by Takeuchi, as a group analogue of a symmetric set \cite{Takeuchi}.
\end{definition}

\begin{facts}
\label{factLYZ}
(1) \cite{LYZ} If $(G, \sigma)$ is a braided group, then $\sigma$ satisfies the braid
relations and is nondegenerate. So  $(G, \sigma)$ forms a
nondegenerate
braided set in the sense of \cite{ESS}, $(G, \sigma)$ is a symmetric set, whenever $\sigma$ is involutive.

Assume that $(X,r)$ is a nondegenerate braided set, $G = G(X,r)$ and $S=S(X,r)$.

(2) \cite{LYZ} There is unique braiding operator
$r_G: G \times G \longrightarrow G \times G $, such that the restriction of
$r_G$ on $X\times X$ is exactly the map $r$.
Furthermore, $(r_G)^2 = id_{G\times G}$ \emph{iff} $r^2 = id_{X\times X}$, so $(G, r_G)$ is a
symmetric group \emph{iff } $(X,r)$ is a symmetric set.

(3) \cite{GIM08} There is unique braiding operator
$r_S: S \times S \longrightarrow S \times S$, such that the restriction of
$r_S$ on $X\times X$ is exactly the map $r$, so $(S, r_S)$ is a braided monoid, and, in particular, a braided set.
\end{facts}
\begin{definition}
  \label{assocBraidedGroupdef}
We shall refer to the group $(G,r_G)$ as "\emph{the (involutive) braided group
associated to} $(X,r),$" or "\emph{the symmetric group
associated to} $(X,r)$",
see \cite{GI18}.
\end{definition}
\begin{remark}
  \label{remark1}
It is proven in \cite{GIVB} that if $(X,r)$ is a finite symmetric set then the monoid
  $S=(X,r)$ is of $I$-type, satisfies cancellation law and the \"{O}re conditions.
  $X$ is embedded in $S$ and $S$ is embedded in $G=G(X,r)$ which is its group of quotients of $S$.
  By \cite{GIVB}, Theor. 1.4,  $G$ is a Bieberbach group, or in other words, $G$ is \emph{finitely generated, torsion-free and abelian-by-finite}, \cite{Charlap}. Moreover, $G(X,r)$ is a solvable group, \cite{ESS}, see also \cite{GI18}, Remark 3.10. Valuable systematic information on the properties of $S(X,r)$ and $G(X,r)$ can be found in \cite{JO:book}.
  \end{remark}

Let $(X,r)$ be a symmetric set, with $G = G(X,r)$, and let $Y\subset X$
be a nonempty subset of $X$. $Y$ is \emph{$r$-invariant } if $r(Y
\times Y)\subseteq Y \times Y.$
In this case $r$ induces a solution $(Y, r_Y)$, where $r_Y$ is  the restriction  $r_{\mid Y\times
Y}$. $(Y, r_Y)$ is called \emph{the induced solution (on $Y$)}. The set
$Y $
is \emph{a (left)
$G$-invariant subset of} $X$,
if it is invariant under the
left action of $G$.
Right $G$-invariant subsets are defined analogously. Note that
$Y$ is left $G$-invariant \emph{iff} it is right
$G$-invariant, so we shall refer to it simply as \emph{a
$G$-invariant subset}. Each $G$-invariant subset $Y$ of  $X$ is also
$r$-invariant. Given a symmetric group $(G, r)$ one defines $G$-invariant subsets of $G$, and $r$-invariant subsets of $G$ analogously.

Let $(G, r)$ be a symmetric group (i.e. an involutive braided group), and let $\Gamma= \Gamma_l$ be \emph{the kernel of the left action of $G$ upon itself}.
The subgroup $\Gamma$ is called \emph{the socle} of $(G, r)$.
It satisfies the following conditions:  (i) $\Gamma$ is \emph{an abelian normal subgroup of} $G$, (ii) $\Gamma$ is \emph{invariant with respect to the left and the right actions of $G$ upon
itself}, in other words $\Gamma$  is an ideal of the braided group $(G, r)$, see Definition \ref{ideal_def}. In particular, $\Gamma$ is $r$-invariant.
Note that in a symmetric set $(X,r)$ one has $\Lcal_x = \Lcal_y$ \emph{iff}  $\Rcal_x =
\Rcal_y$, therefore
$\Gamma$ coincides with the kernel $\Gamma_r$  of the right action of $G$:
\begin{equation}
\label{Gammaeq}
\Gamma = \Gamma_l=\{a \in G\mid {}^au = u, \; \forall u \in G\} =
\{a \in G\mid u^a =
u, \; \forall u \in G\}= \Gamma_r.
\end{equation}

\begin{definition}
\label{ideal_def}  \cite{GI18} Let $(G, r)$  be an involutive braided group. A subgroup $H< G$ is \emph{an ideal of $G$} if $H$ is a normal subgroup of $G$
which is  $G$-invariant, that is ${}^GH :=\{{}^gh \mid g \in G, h \in h\} \subseteq H$ (or, equivalently, $H^G : =\{h^g \mid g \in G, h \in h\} \subseteq H$).
\end{definition}
 Recall that if $H$ is an ideal of the involutive braided group $(G,r)$, then the
quotient group $\overline{G}=G/H$ has also a canonical
structure of an involutive braided group $(\overline{G}, r_{\overline{G}})$
induced from $(G, r)$, \cite{GI18}.

\begin{remark}
\label{remark2}
It is known that when $(X,r)$ is a solution with braided group $(G, r_G)$, where $G= G(X,r)$,
there is a group isomorphism  $G/\Gamma \simeq \Gcal =\Gcal(X,r)$, moreover, the group $\Gcal(X,r)$ is finite, whenever $X$ is a finite set, \cite{So}.
In this case $\Gamma$ is a normal abelian subgroup of finite index in $G$, so the group $G$ is \emph{abelian-by-finite}.
Furthermore, see \cite{Takeuchi, GI18}, the braided structure on $(G, r_G)$ induces
the structure of a braided group on the quotient $G/\Gamma \simeq \Gcal =\Gcal(X,r)$,  we use notation $(\Gcal, r_{\Gcal})$. The map $r_{\Gcal}$
is involutive, since $r$ is involutive, so $(\Gcal, r_{\Gcal})$ is also a symmetric
group.
\end{remark}

\section{The symmetric group $G(X,r)$ of a finite square-free symmetric set $(X,r)$ and its group algebra k[G]}
\label{subs_symgrG(x,r)forXfinite square-free}
In this section we prove the main results of the paper, Theorems \textbf{A}, \textbf{B} and \textbf{C}. From now on till the end of the paper we shall work assuming Convention-Notation
\ref{notation}.

\subsection{Square-free solutions and their algebraic objects}
Especially interesting with their symmetries are solutions satisfying additional combinatorial properties, such as \emph{square-free solutions}, or solutions with condition \textbf{lri}, or with \emph{cyclic conditions} (defined and studied in \cite{GI96, GI04, GIM08}).
The finite square-free solutions  $(X,r)$ are closely related with a special class of quadratic PBW algebras called \emph{binomial skew polynomial ring}, see Theorem
\ref{theorem1}.

\begin{definition}
\label{binomialringdef}
\cite{GI96} A {\em binomial skew polynomial ring\/} is an $n$-generated quadratic algebra
 $A=\textbf{k} \langle x_1, \cdots , x_n\rangle/(\Re)$ with
precisely $\binom{n}{2}$ defining relations
\[\Re=\{x_{j}x_{i} -
c_{ij}x_{i^\prime}x_{j^\prime} \mid 1\leq i<j\leq n\},\] such that:
(a) For every pair $i, j, \; 1\leq
i<j\leq n$, the relation $x_{j}x_{i} - c_{ij}x_{i'}x_{j'}\in \Re,$
satisfies $c_{ij} \in \textbf{k}^{\times}$, $j
> i^{\prime}$, $i^{\prime} < j^{\prime}$;
(b) Every ordered monomial $x_ix_j,$
with $1 \leq i < j \leq n$ occurs in  the right hand side of some
relation in $\Re$;
(c) $\Re$ is the
{\it reduced Gr\"obner basis\/} of the two-sided ideal $(\Re)$, with
respect to the degree-lexicographic order $<$ on $\langle X \rangle$, induced by $x_1 < x_2 < \cdots < x_n$,  or equivalently
the ambiguities $x_kx_jx_i,$ with $k>j>i$ do not give rise to new
relations in $A.$
By \cite{Bergman} condition
(c) may be
 rephrased by saying that
($c^{\prime}$) $A=\textbf{k}\langle x_1, \cdots, x_n\rangle/(\Re)$ is \emph{a PBW} algebra with a $\textbf{k}$-basis the set of ordered
monomials:
\begin{equation}
\label{N-eq}
\Ncal = \{x_1^{\alpha_1}\cdots x_n^{\alpha_n} \mid  \alpha_i \geq 0, \; 1\leq i \leq n\}.
\end{equation}
\end{definition}
Recall from  Gr\"{o}bner bases theory that  $\Ncal$ is the set of \emph{normal monomials}, and  each monomial   $u\in S= S(X,r)$ has unique
normal
form $u_0\in \Ncal$, \cite{Bergman, GI96}.
The monoid $S$ (considered as a set) can be identified with the set $\Ncal$.

Algebraic and homological properties of binomial skew polynomial rings and their close relation to YBE were studied in \cite{GI96, GI12,
GIVB}, et all. Recall that each binomial skew polynomial algebra $A$ is a free left (resp. right) module over $A_p$ of finite rank $p^n$ and with a free basis
 the set $Y$, so $A$ is left and right Noetherian, see \cite{GI96}.
Moreover,by \cite{GIVB}, $A$ is an Artin-Schelter regular algebra of global dimension $n$, $A$ is Koszul and a domain.

\begin{remark}
\label{fact_square-free_implies_lri_cc}
 \cite{GI04, GIM08}.
Every  square-free solution  $(X,r)$ satisfies
the condition \textbf{lri}:
 \[ \textbf{lri:}\quad
\quad ({}^xy)^x= y={}^x{(y^x)} \;\text{for all} \quad
x,y \in X.\]
In other words, \textbf{lri} holds if and only if
$(X,r)$ is nondegenerate, $\Rcal_x=\Lcal_x^{-1}$ and $\Lcal_x =
\Rcal_x^{-1}$.
In particular, $(X,r)$ is  uniquely determined by the left
action:
\[r(x,y) = (\Lcal _x(y), \Lcal^{-1}_y(x)).\]
\end{remark}

Define the set $X^{\star}$ as
\begin{equation}
  \label{notationX}
X^{\star}:=X\bigcup X^{-1}, \quad  \text{where} \quad  X^{-1}=\{ x^{-1}\mid x
\in X \}.
  \end{equation}
We shall need the following useful lemma which is extracted from Proposition 7.6. in \cite{GI18}. It generalizes the cyclic conditions and \textbf{lri}.
\begin{lemma}
 \label{VIPLemma07}
Let  $(X,r)$ be a symmetric set with  \textbf{lri},
 and $G=G(X,r)$.
Then the following equalities hold in $G$ for all $k \in  \mathbb{N}, \; x \in
 X^{\star},
 \; a \in G$.
 \[ {}^a{(x^a)}= x =({}^ax)^a,\quad {}^{a^{-1}}x= x^a, \quad x^{a^{-1}} ={}^ax.\]
   \[\; (x^a)^{-1}=(x^{-1})^a, \quad ({}^ax)^{-1}= {}^a{(x^{-1})}. \]
  \[
 {}^{a^x}x= {}^ax,\quad
x^{{}^xa}= x^a,\quad
{}^{(a^{(x^k)})}x= {}^ax,  \quad
 x^{({}^{(x^k)}a)}= x^a.
\]
 \begin{equation}
  \label{lri66}
 {}^{a}{(x^k)}= ({}^{a}x)^k, \quad (x^k)^a = (x^a)^k.
 \end{equation}
  \end{lemma}

\subsection{The braided group $G=G(X,r)$}

\begin{lemma}
\label{S-lemma} Let $(X, r)$ be a square-free solution, of order $|X|= n$ and cyclic degree $p=p(X,r)$, $S=S(X,r)$
is the associated monoid, and $(G, r_G)$ is the associated braided group.
\begin{enumerate}
\item
\label{S-lemma1}
For every $v \in S$ there exist monomials $y\in Y$ and $W, W^{\prime} \in \Ncal_p,$  such that $v=y.W = W^{\prime}.y$ hold in $S$.
\item
\label{S-lemma2}
If $y_1, y_2 \in Y,  \;$  $  W_1, W_2 \in \Ncal_p, $ then
the equality $y_1 W_1= y_2 W_2$ holds in $S$ \emph{iff} $y_1 = y_2$ and   $W_1= W_2$ are equalities of (normal) words in $\Ncal$.
\item
\label{S-lemma3} $N_p$ is the free abelian monoid,  with  a set of free generators  $X_p$.
The group $\Fcal_p$ is isomorphic to the free abelian group in $n$ generators, $\Fcal_p$ is the group of quotient of the monoid $N_p$.
\item
\label{S-lemma4}
$\Fcal_p$ and $N_p$ are invariant under the left action of $G$ upon itself.
\item
\label{S-lemma5}
$\Fcal_p$ is a subgroup of  $\Gamma$, where  $\Gamma$ is the socle of $G$ (see (\ref{Gammaeq})).
\end{enumerate}
\end{lemma}
\begin{proof}
Part (1) and (2) of our lemma are analogous to  \cite{GI96}, Lemma 4.19, (1) and (3).
Note that the lemma in \cite{GI96} is proven for $p= n!$ but the argument there needs only the equality $(\Lcal_x)^p= id_X,$ for all $x
\in X$,
which is true whenever $p=p(X,r)$ is the cyclic degree of $(X.r)$.
(3). It is clear that $N_p$ is the free abelian monoid with a set of free generators $X_p$, $N_p$ is (canonically) embedded in $\Fcal_p$ which is its group of quotients.
 So $\Fcal_p$ is isomorphic to the free abelian group generated by $X_p$.
 Every $W \in \Fcal_p$ can be written uniquely as
\begin{equation}
\label{proofeqW}
W = x_1^{(k_1p)} x_2^{(k_2p)} \cdots .x_n^{(k_np)},\quad\text{where  }   k_i \in   \mathbb{Z},
\; 1 \leq i \leq n.
\end{equation}
(\ref{S-lemma4}).
We have to show that $N_p$, and $\Fcal_p$ are invariant under the left action of $G$ upon itself.
 We use equality (\ref{lri66}), in Lemma \ref{VIPLemma07}, in which
we replace $k$ with $p$ to obtain
  \[
 {}^{a}{(x^p)}= ({}^{a}x)^p,  \quad \forall \; x \in
 X^{\star},\; a \in G.
\]
Lemma  \ref{VIPLemma07} implies also that for every pair  $a \in G,$ $x \in X^{\star}$, one has ${}^a{(x^{-1})}= ({}^ax)^{-1}$.
Due to the nondegeneracy, the left action of $G$ on $X$ permutes the elements $x_1, \cdots, x_n$, so there are equalities of sets
\begin{equation}
\label{proofeqW3}
\begin{array}{llr}
{}^{a}{X} &=\{{}^{a}{x_1}, {}^{a}{x_2},\cdots, {}^{a}{x_n} \}= \{x_1, x_2,\cdots, x_n \}= X,  &\forall a \in G\\
{}^{a}{(X_p)} &=\{ {}^{a}{(x_1^p)}, {}^{a}{(x_2^p)},\cdots, {}^{a}{(x_n^p)} \}= \{x_1^p, x_2^p,\cdots, x_n^p \}= X_p, &\forall a \in G\\
{}^{a}{(X^{-1})_p} & = \{{}^{a}{((x_1^{-1})^p)},\cdots, {}^{a}{((x_n^{-1})^p)}\}= \{(x_1^{-1})^p,\cdots, (x_n^{-1})^p\}= X^{-1}_p, &\;\forall a \in G.
\end{array}
\end{equation}
Hence the sets $X_p$
and $X^{-1}_{p}=  \{(x_1^{-1})^p, \cdots, (x_n^{-1})^p\}= \{x_1^{-p}, \cdots, x_n^{-p}\}$ are invariant under the left action of  $G$. Now we apply
\textbf{ML2} to yield
 \[ {}^aW \in N_p,  \; \forall \;  W\in N_p,\; a \in G,\quad \text{and} \quad {}^aW \in \Fcal_p, \; \forall \; W \in \Fcal_p,\; a \in G. \]
 It follows that the submonoid $N_p$ and the subgroup $\Fcal_p$ of $G$ are $G$-invariant.

(\ref{S-lemma5}). We have to show that
$\Fcal_p \leq \Gamma$, where  $\Gamma$ is the socle of $G$.
We shall use induction on the length $m = |a|$ of $a \in G$ (see Definition \ref{algobjectsdef} (ii)) to prove
\begin{equation}
\label{proofeqW0}
{}^Wa = a, \quad \forall\;  W \in \Fcal_p,\; a \in G.
\end{equation}
Every element of $G$ is a word in the alphabet $X^{\star}$. For the base of the induction, $m = 1$, we have to show
\begin{equation}
\label{proofeqW1}
{}^Wx = x, \quad \forall\;  W \in \Fcal,\; x \in X^{\star}.
\end{equation}
Let $x \in X$. Recall that each square-free solution $(X,r)$ satisfies \textbf{lri}, so the hypothesis of Lemma \ref{VIPLemma07} is satisfied.
By the choice of $p$, and condition \textbf{lri} one has $(\Lcal_x)^p= id_X = (\Rcal_x)^p$.
Using Lemma \ref{VIPLemma07} (and \textbf{lri} again) we show that the following equalities hold in $G$, for all $k \in   \mathbb{Z}, \; x,y \in X^{\star}$:
\begin{equation}
\label{proofeqW2}
{}^{(y^p)}x = x = x^{(y^p)}, \quad \;
{}^{(y^{-p})} x = {}^{{(y^p)}^{-1}} x = x^{(y^p)} = x, \quad \;
{}^{(y^{(kp)})} x = x.
\end{equation}
  Suppose $W \in \Fcal_p$ then $W = x_1^{k_1p}x_2^{k_2p} \cdots x_n^{k_np}$, for some $k_i \in   \mathbb{Z}$, so condition \textbf{ML1} and (\ref{proofeqW2}) imply
 (\ref{proofeqW1}) which  gives the base for the induction.
 Suppose ${}^Wa = a$, for all $W \in \Fcal_p$, and all $a \in G, \; |a|\leq m$. Let $u \in G$, has length $|u| =m+1,$ so $u = x.a, x\in X^{\star}, |a|= m.$
 We apply \textbf{ML2}, and the inductive assumption to yield
 \[
{}^W{u}= {}^W{(x.a)} =({}^Wx)({}^{W^x}a) = x. ({}^{W^x}a)= x.a = u.\]
Therefore ${}^W{u}= u, \forall u \in G$, so $W \in\Gamma$, which implies $\Fcal_p \leq \Gamma$.
\end{proof}

\begin{prooftheoremA}
\textbf{(\ref{Thm_Fcal41})}. We shall use some of the facts in Remark \ref{remark1}.
 The monoid $S$ has cancelation law, satisfies \"{O}re conditions, and is embedded in $G= G(X,r)$, which is its
group of quotients. So every element  $u \in G$ has the shape  $u =ab^{-1},$ where $a, b \in S$.
 By Lemma \ref{S-lemma}, part (1),  the elements $a$, and $b$ can be presented as
  $a = y_1 W_1$ and $b= W_2 y_2$, where $y_1, y_2 \in Y,$ and $W_1, W_2 \in \Ncal_p$.
Then $b^{-1} = (y_2)^{-1} W_2^{-1},$ where $W_2^{-1}\in \Fcal_p$.
First we shall present $(y_2)^{-1}$ in the form $(y_2)^{-1}= t W^{\prime}$,   with  $t \in S,\; W^{\prime}\in \Fcal_p$.
Since $y_2\in Y$ one has
\begin{equation}
\label{proofeq0}
y_2 = x_1^{\alpha_1}\cdots x_n^{\alpha_n},\quad\text{where  } 0 \leq \alpha_i \leq p-1,\; 1 \leq i \leq n.
\end{equation}
Note that for any $x \in X$, and any integer $\alpha$ with $0 \leq \alpha \leq p-1$, there is an obvious  equality
\[
x^{-\alpha}  = x^{p-\alpha}x^{-p}= x^{\beta}x^{-p},\; \text{where  } \;    0 \leq \beta = p-\alpha\leq p-1.
\]
This, together with Lemma \ref{S-lemma}, part (1) imply
\[
\begin{array}{lll}
y_2 ^{-1}
 &= x_n^{-\alpha_n}\cdots x_1^{-\alpha_1} =  (x_n^{p-\alpha_n}. x_n^{-p}) \cdots (x_1^{p-\alpha_1}. x_1^{-p})&\\
   &=  (x_n^{\beta_n}x_n^{-p})  \cdots (x_1^{\beta_1}x_1^{-p})\;
   =   (x_n^{\beta_n}  \cdots x_1^{\beta_1})W^{\prime} & =  t\;W,  \quad t \in Y, \;  W\in \Fcal_p.
 \end{array}
\]
Therefore
\[
b^{-1} = y_2^{-1} W_2^{-1}= (t.W)W_2^{-1}
  = t (W.W_2^{-1}) \; = t W_3, \;\; t \in Y, \;  W_3 =W.W_2^{-1}\in \Fcal_p.
\]
The following equalities hold in $G$:
\[
\begin{array}{rll}
ab^{-1} &= (y_1W_1)(tW_3) = y_1(W_1t)W_3
  = y_1 ({}^{W_1}t)((W_1)^t) W_3,\quad &  \; (W_1)^t,   W_3\in \Fcal\\
  &= (y_1t) W_4, \;\; \quad \quad W_4  = ((W_1)^t)W_3 \in \Fcal_p, \; & \text{since }\; {}^{W_1}t =t\\
  &=  (y W_0)W_4, \;   \;  y W_0:= y_1t,  \;\; y\in Y, \; W_0 \in \Ncal_p,\quad & \text{by Lemma \ref{S-lemma}}\\
  &=  y W,   \quad   \quad\; \; y\in Y, \; W: = W_0W_4 \in \Fcal_p.&
 \end{array}
\]
We have shown that every element $u \in G$ can be presented as
$
u = y W, \; y\in Y, \; W\in \Fcal_p.
$
We claim that this presentation is unique. Suppose
$y_1 W_1 = y_2 W_2$ holds in $G$, where $y_1, y_2 \in Y,$  and $W_1, W_2 \in \Fcal_p.$
Every element $W \in \Fcal_p$ has a presentation as $W = uv^{-1},$ where
$u,v \in N_p$. Without loss of generality we may assume  $u,v \in \Ncal_p$.
We use the equalities
$W_1 = u_1v_1^{-1}, \; W_2 = u_2v_2^{-1}$, where $u_i, v_i\in \Ncal_p, i = 1,2$,
and obtain $y_1 u_1v_1^{-1} =y_1W_1  = y_2W_2 = y_2u_2v_2^{-1}$, hence
\[y_1 u_1v_1^{-1} = y_2u_2v_2^{-1}.\]
The elements of $\Fcal_p$ is commute, hence multiplying both sides of this equality by $v_1v_2$, we obtain
\[
y_1 (u_1v_2) = y_2(u_2v_1),  \; \; \text{where} \; y_1, y_2 \in Y, \;\text{  and  } \; u_1v_2,\; u_2v_1\in N_p.
\]
It follows then from Lemma \ref{S-lemma} that the equalities $y_1 = y_2$, and  $u_1v_2 = u_2v_1$ hold in $N_p $, and in
$\Fcal_p$.
Next we multiply both sides of $u_1v_2 = u_2v_1$ by $v_1^{-1}v_2^{-1}$, to
obtain  $W_1 = u_1v_1^{-1} = u_2v_2^{-1}= W_2\in \Fcal_p$, which proves part (\ref{Thm_Fcal41}).
Parts (\ref{Thm_Fcal42}) and (\ref{Thm_Fcal43}) are straightforward, in particular
 (\ref{disjuneq}) is in force. $\quad \quad \quad \quad\quad \quad \quad \quad \quad \quad \quad\quad \quad\quad\Box$
\end{prooftheoremA}

\begin{prooftheoremB}
\textbf{(\ref{Thm_Fcal2a})}.  We have to show that $\Fcal_p $ is an ideal of $G$ that is, a $G$-invariant normal subgroup of $G$.
By Lemma \ref{S-lemma}  $\Fcal_p$ is a $G$-invariant subgroup of $\Gamma$. It remains to prove
that $\Fcal_p$ is a normal subgroup of $G$.
Suppose $W \in \Fcal_p, a \in G$.
We use the compatibility condition  \textbf{M3}, and the equality $a^W= a$ (since $\Fcal_p < \Gamma$)  to yield:
\[(aW)a^{-1}= ({}^aW)(a^W)a^{-1}=({}^aW) (aa^{-1}) = {}^aW \in \Fcal_p. \] So  $\Fcal_p$ is a normal subgroup of $G$, and therefore it is an ideal of $G$. Clearly,
$\Fcal_p \;\triangleleft \;\Gamma\;\triangleleft\; G$ is a chain of ideals in the braided group $G$.
\textbf{(\ref{Thm_Fcal06a})}. We have shown that $\Fcal_p $ is an ideal of $G$, hence the quotient group $\widetilde{G}= G/\Fcal_p$ is an involutive braided group of order
$|\widetilde{G}|=[G: \Fcal_p].$
It follows straightforwardly from Theorem \textbf{A} that
 the index of $\Fcal_p$ in $G$ is exactly
 $[G: \Fcal] =|Y| = p^n,$ hence $|\widetilde{G}|= p^n$.
 Moreover, \cite{GI18}, the canonical projection
$\pi: G \longrightarrow \widetilde{G}= G/\Fcal_p, \quad g = y W \mapsto \tilde{g}=\tilde{y},$ induces the structure of an involutive braided group
on the finite set \[\widetilde{G} = \tilde{Y} :=\{\tilde{y}\mid y \in Y\},\]
and the induced operation "$\cdot$" on $\tilde{Y}$ satisfies $(\ref{operationeq})$. It is clear that the restriction
 $\pi_{|Y} : Y \longrightarrow \tilde{Y} =\widetilde{G}$ is a bijective map, in other words $Y$ is embedded in $\widetilde{G}$, and so is $X \subset Y$.
 We identify $X$ with its image $\tilde{X}$ in $\widetilde{G}$. Clearly the set $\tilde{X}$ generates  $\widetilde{G}$. We have proven parts (\ref{Thm_Fcal06a}) and  (\ref{Thm_Fcal06b}). \textbf{(\ref{Thm_Fcal06c}).}
Consider the chain of ideals  $\Fcal_p \;\ra \;\Gamma \;\ra \;G.$
The 3rd isomorphism theorem for braided groups, see  \cite{GI18}, Remark 4.13, and the well-known group isomorphism   $\Gcal\simeq (G/\Gamma)$ imply
\[ \Gcal \simeq (G/\Gamma) \simeq (G/\Fcal_p)/(\Gamma/\Fcal_p),\;\text{and }\quad |G:\Gamma||\Gamma: \Fcal_p|=|G: \Fcal_p|= p^n. \]
It follows that the order $|\Gcal|=|G:\Gamma|$ divides
 $p^n$. In particular, if $p$ is a prime, then  the permutation group $\Gcal$
is a $p$-group.    $\quad \quad \quad \quad\quad \quad \quad \quad \quad \quad\quad \quad \quad \quad \quad\quad\Box$
\end{prooftheoremB}
\begin{corollary}
Let $(X,r)$ be a finite square-free solution of order $|X|=n$ with cyclic degree $p=p(X,r)$.
Then there exists a finite involutive braided group $(B, \sigma)$ of order $p^n$, such that $X$ is embedded in $B$ and generates the group $B$.
Moreover, $X$ is $B$-invariant, and $r=\sigma_{|X\times X}$ is the restriction of $\sigma$ on $X\times X$.
\end{corollary}
Theorem \textbf{C} (4) shows that the group ring $\textbf{k}[G]$ is left and right Noetherian which implies the following.
\begin{corollary}
\label{Thm_Fcal5a}  $G$ is a Noetherian group, that is every subgroup of $G$ is finitely generated.
\end{corollary}

The following result is well-known, it can be extracted from Remarks \ref{remark1} and \ref{remark2}, but here it is an independent and straightforward consequence from Theorem B.
\begin{corollary}
\label{Thm_Fcal}
Under the hypothesis of Theorem \textbf{B}. The group $G=G(X,r)$ is a finitely-generated torsion-free, abelian-by-finite solvable group, and therefore $G(X,r)$ is torsion-free
\emph{poly-$\mathbb{Z}$-by-finite}.
\end{corollary}

\subsection{The group algebra  $\textbf{k}[G]$ over a field \textbf{k}}
One of the famous and still open problems about abstract
group rings is the \emph{Kaplansky zero divisor conjecture} which is closely related to two more conjectures, the unit conjecture (U) and the idempotent conjecture (I) given below.
\begin{conjectures}
(I. Kaplansky)
Let $\textbf{k}$ be a field, let $G$ be a torsion-free group.
\begin{enumerate}
\item[(Z)]
The group ring $\textbf{k}[G]$  is a domain.
\item[(I)]
$\textbf{k}[G]$ does not contain any non-trivial idempotents: if $a^2 = a$, then $a = 1$ or $a = 0$.
\item[(U)]
$\textbf{k}[G]$ has no non-trivial units: if $ab = 1$ in $\textbf{k}[G]$, then $a = \alpha.g$ for some $\alpha \in \textbf{k}^{\times}$ and $g \in G$.
\end{enumerate}
\end{conjectures}
It is well-known that $(U)\Longrightarrow (Z)\Longrightarrow (I)$. There are numerous positive results, in particular, the zero divisor conjecture is known to hold for all virtually solvable group.
\begin{remark}
    \label{Thm_Fcal2G06}
It is known that the group $G=G(X,r)$ of every finite solution $(X,r)$ satisfies (I), the zero-divisor conjecture. This follows from the works \cite{Farkas}, and \cite{KLM}, see for details the proof of Lemma \ref{VIPprop}.
 Moreover, in the special case when $(X,r)$ is a finite multipermutation solution,
   \cite{Chouraqui16}, Theorem 2, states that the group $G(X,r)$ is left orderable, and therefore $G$  satisfies Kaplansky's unit conjecture (U).
\end{remark}
 \begin{remark}
    \label{Thm_Fcal2aa}
In usual notation and conventions, for $S = S(X,r), G=G(X,r)$, the canonical embedding $S \hookrightarrow G$ induces a natural embedding  of the semigroup algebra $A = \textbf{k}S$ into the group algebra $\textbf{k}[G]$. One has $A= \textbf{k}S \hookrightarrow \textbf{k}[G]\hookrightarrow \mathcal{Q}(A),$ where $\mathcal{Q}(A)$ is the canonical right quotient ring of the domain $A$.
\end{remark}
\begin{lemma}
\label{VIPprop}
 Let $G= G(X,r)$ be the braided group of a finite nondegenerate symmetric set $(X,r)$, let $\textbf{k}$ be a field. Then the group algebra \textbf{k}[G] is a left (respectively right) noetherian domain.
Moreover, $\textbf{k}[G]$ has finite global dimension, whenever the field $\textbf{k}$ has characteristic $0$.
\end{lemma}
\begin{proof}
We discussed, see Remarks \ref{remark1} and \ref{remark2}, that the braided group $G=G(X,r)$ is a torsion-free finitely generated
abelian-by-finite solvable  group. Therefore $G$ is polycyclic. In fact, the group $G$ is (poly-$\mathbb{Z}$)-by-finite.
Now we simply apply several already classical results from the theory of group rings.
Hall proved that the group ring $\textbf{k}[G]$  of a polycyclic group $G$ is right (resp. left) Noetherian, \cite{Hall} Theorem 1.
Farkas and Snider,  showed that if $G$ is a torsion free, polycyclic-by-finite group and  $\textbf{k}$ is a field of characteristic $0$, then (1) the group ring $\textbf{k}[G]$ has no zero divisors, see \cite{Farkas}, Main Theorem; and (2) $\textbf{k}[G]$ has finite global dimension, \cite{Farkas}, Lemma 1 ("Syzygy" theorem) .
More generally, Kropholler, Linell and Moody, prove the zero divisor conjecture for a large class of torsion-free groups which contains all torsion-free solvable-by-finite groups, and for an arbitrary field $\textbf{k}$, see \cite{KLM}, Theorem 1.3.
It follows from the above discussion that for every field $\textbf{k}$ the group ring $\textbf{k}[G]$ is a left (respectively right) Noetherian domain. $\textbf{k}[G]$
has finite global dimension, whenever $\textbf{k}$ is a field of characteristic $0$.
\end{proof}

\begin{prooftheoremC}
 \textbf{ (\ref{Thm_Fcal4}).}
$\Fcal_p$ is the free abelian group in $n$-generators $x_1^p, \cdots, x_n^p$, then it is well-known the group ring $\textbf{k}[\Fcal_p]$ is a Noetherian domain.
 It follows from Theorem A that the group algebra
$\textbf{k}[G]$ is a left module of finite rank
over the algebra $\textbf{k}[\Fcal_p]$, generated by $Y= \{y_i\mid 1 \leq i \leq p^n\}$, and therefore
$\textbf{k}[G]$ is left Noetherian. Analogously one shows that $\textbf{k}[G]$ is right Noetherian.
We have to prove that $\textbf{k}[G]$ is \emph{a free left} $\textbf{k}[\Fcal_p]$-\emph{module with a free basis} $Y$.
To simplify notation set $q = p^n$, and assume that there is a relation:
\begin{equation}
\label{proofeq06}
g_1 y_1 +g_2 y_2 +\cdots + g_qy_q =0, \quad \text{where} \; g_i \in \textbf{k}[\Fcal_p], \; 1 \leq i \leq q.
\end{equation}
But $\{g_1, \cdots, g_s\} \subset \textbf{k}[\Fcal_p]$ is a finite set, and each of its elements is (a finite) linear combination of words $W \in \Fcal_p$, therefore one can find a monomial $W_0 \in N_p$, such that
$W_0. g_i = a_i \in \textbf{k} S_p= A_p, \;1 \leq i\leq q$. We multiply (\ref{proofeq06}) (on the left) by $W_0$, and obtain the following relation in the monoidal algebra
$A= \textbf{k} S$:
\begin{equation}
\label{proofeq07} a_1 y_1 +a_2 y_2 +\cdots + a_qy_q =0, \quad \text{where} \; a_i \in \textbf{k}[S_p], \; 1 \leq i \leq q.
\end{equation}
We have proven that the algebra $A$ is a free left module over $A_p= \textbf{k}S_p$ with a free-basis $Y$, see \cite{GI96}, Theorem I, (for more details see the proof of Lemma (4.20)).
Therefore $a_i= 0, \; 1 \leq i \leq q$. But $a_i= W_0.g_i$, for all $i$'s, and since $\textbf{k}[G]$ has no zero divisors (by Lemma Lemma \ref{VIPprop}), we obtain $g_i= 0,\; 1 \leq i \leq q$.
It follows that $Y$ is a free basis of $\textbf{k}[G]$, considered as a left $\textbf{k}[\Fcal_p]$- module.
  \textbf{(\ref{Thm_Fcal2}).} Recall that a polynomial in n variables $f(X_1, X_2,\cdots , X_n)\in \textbf{k}[X_1, X_2,\cdots , X_n]$
is \emph{a symmetric polynomial} if for any permutation $\sigma$ of the subscripts $1, 2,\cdots, n$ one has
$f(X_{\sigma(1)}, X_{\sigma(2)}, \cdots, X_{\sigma(n)}) = f(X_1, X_2,\cdots , X_n)$.
It is well-known that any symmetric polynomial can
be presented as a polynomial expression with coefficients from $\textbf{k}$ in the first $n$ \emph{power sum symmetric polynomials}
$ s_k(X_1, X_2,\cdots , X_n), 1\leq k \leq n$, where
 \[
 s_k(X_1, X_2,\cdots , X_n)= X_1^k+X_2^k + \cdots + X_n^k,\quad k \geq 1.
 \]
In particular, the remaining power sum polynomials $s_k(X_1, X_2,\cdots , X_n)$  for $k > n$, can be also expressed in the first $n$ power sum polynomials.
So, the set of all symmetric polynomials in $\textbf{k}[X_1, X_2,\cdots , X_n]$ is an $n$-generated subalgebra $\mathfrak{S}= \textbf{k}\langle s_1 \cdots, s_n\rangle$
of the commutative polynomial ring $\textbf{k}[X_1, X_2,\cdots , X_n]$.
We know that the subalgebra $A_p =\textbf{k}[x_1^p, x_2^p,\cdots , x_n^p] \subset \textbf{k} S$ is isomorphic to the commutative polynomial ring
$\textbf{k}[X_1, X_2,\cdots , X_n]$ in $n$ variables (we may consider the isomorphism extending the assignment $x_i^p \rightarrow X_i, 1 \leq i \leq n$).
Consider now the subalgebra $\mathfrak{S}= \mathfrak{S}(x_1^p, \cdots, x_n^p)$ consisting of all symmetric polynomials $f(x_1^p, \cdots, x_n^p)\in A _p$.
$\mathfrak{S}$ is generated (as a $\textbf{k}$-algebra) by the power sums
\[s_k= s_k((x_1)^p, \cdots , (x_n)^p)= (x_1^p)^k + \cdots + (x_n^p)^k =(x_1)^{kp} + \cdots +(x_n)^{kp}. \]
To show that $\mathfrak{S}$ is in the center of $\textbf{k}[G]$ it suffices to verify the following equalities for all $y \in X, \; 1 \leq k \leq n$:
\[
\begin{array}{l}
s_k(x_1^p, \cdots , x_n^p). y = y. s_k(x_1^p, \cdots , x_n^p), \quad
s_k(x_1^p, \cdots , x_n^p).(y^{-1}) = (y^{-1}). s_k(x_1^p, \cdots , x_n^p).
\end{array}
\]
By the nondegeneracy of the right action there is an equality of sets $\{x_1^y,\cdots , x_n^y\}= X, \forall y \in X$. The following equalities hold in $\textbf{k}[G]$:
\[
\begin{array}{lll}
s_k .y& = ((x_1)^{kp} + \cdots +(x_n)^{kp}).y= ((x_1)^{kp}).y + \cdots +((x_n)^{kp}).y& \\
     &= ({}^{x_1^{kp}}y).((x_1)^{y})^{kp}+\cdots +({}^{x_n^{kp}}y).((x_n)^{y})^{kp}\;&\text{by \textbf{M3}}\\
      &= (y).((x_1)^{y})^{kp}+\cdots + (y).((x_n)^{y})^{kp}& \text{by}\;{}^{x^{kp}}y = y, \forall x \in X\\
      &= y.(((x_1)^{y})^{kp}+\cdots + ((x_n)^{y})^{kp})& \\
      &=y.((x_1)^{kp}+\cdots + ((x_n)^{kp}))= y .s_k,&
\end{array}
\]
which verify $s_k.y = y.s_k$, for all $k \geq 1$ and all $y \in X.$
In particular,  this implies that the subalgebra $\mathfrak{S}$ is in the centre of the monoidal algebra $A= \textbf{k} S$.
The proof of the equality  $s_k .(y^{-1})= (y^{-1}).s_k, \; y \in X,\; k \geq 1$, is analogous. In this case one uses the additional relations
given by Lemma \ref{VIPLemma07}, more specifically
${}^a {(y^{-1})}=({}^ay)^{-1}$
and $a^{(y^{-1})} = {}^ya, \; \forall y \in X,\; a \in G$. This proves part (\ref{Thm_Fcal2}).
 Part  (\ref{Thm_Fcal5}) follows from Lemma \ref{VIPprop}. Note that we have shown in part (1)  that
 $\textbf{k}[G]$ is left and right Noetherian, independently from Lemma \ref{VIPprop}.
    $\quad \quad \quad \quad \quad \quad \quad \quad \quad\quad \quad\quad \quad\quad \quad\quad \quad\quad \Box$
   \end{prooftheoremC}

{\bf Acknowledgments}. This paper was written during my visit at the Max Planck Institute for Mathematics in the Sciences (MiS) in Leipzig, in 2018-2019.
It is my pleasant duty to thank MiS, and Bernd Sturmfels for the support and the creative, inspiring
atmosphere. I thank all colleagues and friends, the people from the wonderful MiS library, and Saskia Gutzschebauch for the cooperation.

\end{document}